\RequirePackage{fix-cm}
\documentclass[smallextended]{svjour3}       % onecolumn (second format)
\smartqed  % flush right qed marks, e.g. at end of proof
\usepackage{amsfonts,color,graphicx}
\usepackage{amsmath}
\def\ic{\textrm{i}}
\def\h{\hspace{0cm}}
\begin{document}

\title{A modified version of the PRESB preconditioner  for a class of  non-Hermitian  complex systems of linear equations}

\titlerunning{Modified  PRESB preconditioner}        % if too long for running head

\author{Owe Axelsson, Davod Khojasteh Salkuyeh}


\institute{
	O. Axelsson \at
	Department of Information Technology, Uppsala University, Uppsala, Sweden\\
    \email{owe.axelsson@it.uu.se} \\
    \\
 \\	
	D. K. Salkuyeh \at
              Faculty of Mathematical Sciences, University of Guilan, Rasht, Iran\\
              Center of Excellence for Mathematical Modelling, Optimization and Combinational
              Computing (MMOCC), University of Guilan, Rasht, Iran\\
              \email{khojasteh@guilan.ac.ir}           %  \\
}

\date{Received: date / Accepted: date}

\maketitle

\begin{abstract}
We present a modified version of the PRESB preconditioner  for two-by-two block system of linear equations  with the coefficient matrix 
$$\textbf{A}=\left(\begin{array}{cc}
	F & -G^* \\
	G & F
\end{array}\right),$$ 
where  $F\in\mathbb{C}^{n\times n}$ is Hermitian positive definite and $G\in\mathbb{C}^{n\times n}$ is positive semidefinite. Spectral analysis of the preconditioned matrix is analyzed. In each iteration of  a  Krylov subspace method, like GMRES, for solving the preconditioned system in conjunction with proposed preconditioner two subsystems with Hermitian positive definite coefficient matrix should be solved which can be accomplished exactly using the Cholesky factorization or inexactly utilizing  the conjugate gradient method.  Application of the proposed preconditioner to the systems arising from finite element discretization of  PDE-constrained optimization problems is presented. Numerical results are given to demonstrate   the efficiency of the preconditioner.
	
\end{abstract}

\section{Introduction}
We are concerned with the following two-by-two block system of linear equations (see \cite{Lukas,Zeng,Zheng})
\begin{equation}\label{Eq01}
	\textbf{A}\textbf{x}=
	\left(\begin{array}{cc}
		F & -G^* \\
		G & F
	\end{array}\right)
	\left(\begin{array}{cc}
		x \\ y
	\end{array}\right)=
	\left(\begin{array}{cc}
		p\\ q
	\end{array}\right)=\textbf{b},
\end{equation}
where $F\in\mathbb{C}^{n\times n}$ is Hermitian positive definite (HPD),  $G\in\mathbb{C}^{n\times n}$ is positive semidefinite (PSD) (in general, non-Hermitian),  and  $p,q,x,y\in\mathbb{C}^{n}$  in which $x,y$ are unknowns to be computed.  We recall that the matrix $A\in \mathbb{C}^{n\times n}$ is said to be positive definite (PD)  if  $\Re(x^*Ax)>0$, for every nonzero $x\in \mathbb{C}^n$ (see \cite{BaiBook}). Here, the real part of  a complex number $z$ is denoted by  $\Re(z)$.  Similarly, a matrix $A\in \mathbb{C}^{n\times n}$ is said to be  positive semidefinite (PSD) if  $\Re(x^*Ax)\geq 0$, for every $x\in \mathbb{C}^n$. It is straightforward to prove that the matrix  $A\in \mathbb{C}^{n\times n}$ is PD (resp. PSD) if and only if the matrix $A+A^*$ is HPD (resp. Hermitian positive semidefinite). Naturally, a positive definite matrix which is also Hermitian, is called Hermitian positive definite (HPD). Similarly, a Hermitian positive semidefinite (HPSD) matrix is defined.

We assume that the matrix $\textbf{A}$ is large and sparse. So using  direct methods such as Gaussian elimination to solve the above system can be computationally expensive due to the large number of operations required.  Therefore, for large and sparse systems of linear equations, iterative methods are often the method of choice for obtaining a solution efficiently.

For two matrices  $F,G\in\mathbb{R}^{n\times n}$, and vectors $x,y,p,q \in \mathbb{R}^n$, the complex system
\begin{equation}\label{Eq02}
	(F+\ic G)(x+\ic y)=p+\ic q,
\end{equation}
where $\ic=\sqrt{-1}$ is the imaginary unit, can be equivalently written  in the real form as  
\begin{equation}\label{Eq02New}
	\left(\begin{array}{cc}
		F & -G \\
		G & F
	\end{array}\right)
	\left(\begin{array}{cc}
		x \\ y
	\end{array}\right)=
	\left(\begin{array}{cc}
		p\\ q
	\end{array}\right).
\end{equation}
As we see, in this case, if the matrix $G$ is symmetric positive semidefinite and $F$ is symmetric positive definite, then the  system \eqref{Eq01} is a real equivalent form of  system \eqref{Eq02}.
Equations of the form (\ref{Eq02}) are commonly encountered in scientific computing and engineering applications. For instance, they arise in the numerical solution of the Helmholtz equation and time-dependent partial differential equations (PDEs) \cite{Bertaccini}, diffuse optical tomography \cite{Arridge}, algebraic eigenvalue problems \cite{Moro, Schmitt}, molecular scattering \cite{Poirier}, structural dynamics \cite{Feriani}, and lattice quantum chromodynamics \cite{Frommer}.

In the case when $F$ and $G$ are respectively  symmetric positive definite (SPD) and  symmetric positive semidefinite (SPSD), there are several methods for solving Eqs. \eqref{Eq02} and \eqref{Eq02New}. 
For example, based on the  Hermitian/skew-Hermitian splitting (HSS) method \cite{Bai1}, Bai et al. in \cite{Bai2}   presented a modified version of the HSS iterative method, called MHSS,  to solve systems of the form (\ref{Eq02}).  Next, Bai et al. in \cite{Bai3} established a preconditioned version of the MHSS method for solving the system \eqref{Eq02New}.   Salkuyeh et al. in \cite{Salkuyeh1} solved the system (\ref{Eq02New}) by the generalized successive overrelaxation (GSOR) iterative method and then Hezari et al. in \cite{Hezari-NLA} proposed a preconditioned version of the GSOR method.  
The scale-splitting (SCSP) iteration method for solving  (\ref{Eq02}) was presented by Hezari et al. in \cite{SCSP}. 
Using the idea of the SCSP iteration method, Salkuyeh in \cite{TSCSP} set up a two-step scale-splitting (TSCSP) for solving Eq. (\ref{Eq02}) and then Salkuyeh and Siahkolaei in \cite{TTSCSP} introduced the two-parameter version of the TSCSP method. 
A similar method to TSCSP, called Combination method of Real part and Imaginary part (CRI),  was presented by Wang et al. in \cite{CRI}.
The transformed matrix preconditioner (TMP) was presented by Axelsson and Salkuyeh in \cite{Axe4}. In a sequence of  papers, Axelsson et al. 
presented the PRESB (Preconditioned Square Block) preconditioner for the system \eqref{Eq02New} which is written as
\begin{equation}\label{Psym}
	\mathbf{P}=		\left(\begin{array}{cc}
		F & -G \\
		G & F+2G
	\end{array}\right).
\end{equation}
The PRESB preconditioner has  favorable properties\cite{Axe1,Axe2,Axe3,Axe4,Lukas,Owe-Janos,AKN}. The eigenvalues of the preconditioned matrix $\mathbf{P}^{-1}\mathbf{A}$ are clustered in $[\frac{1}{2},1]$. It is well-known  that the rate of convergence of the conjugate gradient (CG) method depends on eigenvalue distribution of the coefficient matrix, because  the spectral condition number of the coefficient matrix is the ratio of its largest eigenvalue to the smallest one. However, as mentioned for example in \cite{BenziJCP}, for
non-symmetric (non-normal) problems  the eigenvalues may not describe the convergence of the iterative methods. Nevertheless, a clustered spectrum (away from 0) often leads to rapid convergence. For more details see \cite{AxK,BenziFV,EES}.
So the Krylov subspace methods like GMRES would be suitable for solving the preconditioned  system $\mathbf{P}^{-1} \mathbf{A}\mathbf{x}=\mathbf{P}^{-1}\mathbf{b}$. It can be seen that in each iteration of a Krylov subspace method, a linear system of equations with the coefficient matrix $\mathbf{P}$ should be solved, which can be accomplished by solving two subsystems with the coefficient matrix   $F+G$. Since the matrix $F+G$ is HPD, the corresponding system can be solved exactly using the Cholesky factorization or inexactly using the CG method.   

In the case that $F$ is HPD and $G$ is non-Hermitian positive definite, the PRESB preconditioner to the system \eqref{Eq01} is modified as
	\begin{equation}\label{Eq05}
		\mathbf{Q} = \left(\begin{array}{cc} 
			F  &  -G^* \\ 
			G  &  F+G+G^* \end{array}\right).
	\end{equation}
	Theoretical results given in \cite{Owe-Janos} show that, in this case, the  
	eigenvalues of $\mathbf{Q}^{-1}\mathbf{A}$ are included in the interval $[\frac{1}{2},1]$. It is not difficult  to see that 
	the matrix $\mathbf{Q}$ can be decomposed as
	\begin{equation}\label{Eq07a}
		\mathbf{Q} = \left(\begin{array}{cc} I & -I \\ 0 & I \end{array}\right)
		\tilde{\mathbf{Q}}
		\left(\begin{array}{cc} I & I \\ 0 & I \end{array}\right),
	\end{equation}
	where
	$$
	\tilde{\mathbf{Q}}=\left(\begin{array}{cc} F+G & 0 \\ G & F+G^*\end{array}\right).
	$$
	Considering  Eq. \eqref{Eq07a}, we  see that in the implementation of the preconditioner $\textbf{Q}$ in a Krylov subspace method, we have to solve two subsystems with the coefficient matrices $F+G$ and $F+G^*$. Obviously, these matrices are non-Hermitian  positive definite and we can not employ the CG method for solving the corresponding systems.
	To overcome this drawback, based on the PRESB preconditioner we propose a preconditioner to  the system \eqref{Eq01} with the same implementation properties of  the preconditioner $\mathbf{P}$ given in \eqref{Psym}. Hereafter, the proposed preconditioner is called MPRESB for ``modified PRESB''.

Throughout the paper we use the following notations. For a complex number $z$, the imaginary part of $z$ is denoted by $\Im(z)$. 
For the imaginary unit, we use $\rm{i}=\sqrt{-1}$.
For a matrix $A\in\mathbb{C}^{m\times n}$, we use $A^*$ for the conjugate transpose of $A$. 
We use $\textrm{cond}(A)$ for the condition number of a  nonsingular  matrix $A$ which is defined by $\textrm{cond}(A)=\|A\|\|A^{-1}\|$, in which $\|.\|$ is a matrix norm.  For an square matrix $A$ with real eigenvalues, the largest and the smallest eigenvalues of $A$ are denoted by $\lambda_{\max}(A)$ and $\lambda_{\min}(A)$, respectively. For two vectors $x,y\in\mathbb{C}^{n}$, the \textsc{Matlab} notation $[x;y]$ is used for $[x^*,y^*]^*$.

This paper is organized as follows. In the next section we extend the PRESB preconditioner to the case when the matrix $G$ is non-Hermitian. Application of the proposed preconditioner to a model problem is given in Section \ref{Sec3}. Numerical experiments are presented in Section \ref{Sec4}. Finally, concluding remarks are given in Section \ref{Sec5}

\section{The MPRESB preconditioner}\label{Sec2}

We propose the preconditioner 
$$
\mathbf{R} = \left(\begin{array}{cc}
	F & -\frac{1}{2}\left(G+G^*\right) \\ \frac{1}{2}\left(G+G^*\right) & F + \left(G+G^*\right)
\end{array}\right),
$$
to the system \eqref{Eq01}.  This matrix can be factorized as
\begin{equation}\label{Eq08}
	\mathbf{R} = 
	\left(\begin{array}{cc} I & -I \\ 0 & I \end{array}\right)
	\tilde{\mathbf{R}}
	\left(\begin{array}{cc} I & I \\ 0 & I \end{array}\right),
\end{equation} 
in which 
\[
\tilde{\mathbf{R}}=\left(\begin{array}{cc} F+\frac{1}{2}\left(G+G^*\right) & 0 \\ \frac{1}{2}\left(G+G^*\right) & F+\frac{1}{2}\left(G+G^*\right) \end{array}\right)
\]
Hence, in the implementation of the preconditioner $\mathbf{R}$
two subsystems with the coefficient matrix  $F+\frac{1}{2}\left(G+G^*\right)$  should be solved. Since this matrix is HPD, we can use the CG method or the Cholesky method to solve the corresponding system. 

\begin{theorem}\label{Thm2}
	If $\lambda$ is an eigenvalue of $\mathbf{R}^{-1}\mathbf{Q}$, then 
	$\Re(\lambda)=1$
	and
	$$
	\Im (\lambda) = \pm\frac{1}{2\ic} \frac{x^* \left(G^* - G\right) x}{x^* \left(F+\frac{1}{2}\left(G+G^*\right)\right)x},
	$$
	where $x\in \mathbb{C}^{n}$ is a nonzero vector.
\end{theorem}

\begin{proof}
	Using Eqs. \eqref{Eq07a} and  \eqref{Eq08} we get
	\[
	\mathbf{R}^{-1} \mathbf{Q} = 
	\left(\begin{array}{cc} I & I \\ 0 & I \end{array}\right)^{-1}
	\tilde{\mathbf{R}}^{-1} \tilde{\mathbf{Q}}
	\left(\begin{array}{cc} I & I \\ 0 & I \end{array}\right).
	\]
	This relation shows that the matrices $\mathbf{R}^{-1} \mathbf{Q}$ and $\tilde{\mathbf{R}}^{-1} \tilde{\mathbf{Q}}$ are similar and their eigenvalues are the same. On the other hand, it is easy to see that  $\tilde{\mathbf{R}}^{-1} \tilde{\mathbf{Q}}$	is of the form
	\[
	\tilde{\mathbf{R}}^{-1} \tilde{\mathbf{Q}}=
	\left(\begin{array}{cc} \left(F+\frac{1}{2}\left(G+G^*\right)\right)^{-1}(F+G) & 0 \\ 
		Z & \left(F+\frac{1}{2}\left(G+G^*\right)\right)^{-1}\left(F+G^*\right) 
	\end{array}\right),
	\]
	in which $Z$ is an 	$n \times n $ matrix. Hence, the spectrum of  $\mathbf{R}^{-1} \mathbf{Q}$ is the union of the spectrum of  matrices $W$ and $S$, where
	\begin{eqnarray}
		W &=& \left(F+\frac{1}{2}\left(G+G^*\right)\right)^{-1}(F+G),\\
		V &=& \left(F+\frac{1}{2}\left(G+G^*\right)\right)^{-1}(F+G^*).
	\end{eqnarray}
	Now let  $(\lambda,x)$ be an eigenpair of  $W$. Then, we have
	\[
	\left(F+G\right)x=\lambda \left(F+\frac{1}{2}\left(G+G^*\right)\right)x,
	\] 
	or
	\[
	\left( F+\frac{1}{2}\left(G+G^*\right)+\frac{1}{2}\left(G-G^*\right)\right)x=\lambda \left(F+\frac{1}{2}\left(G+G^*\right)\right)x.
	\]
	Premultiplying both sides of this equation by $x^*$  we arrive at
	$$
	\lambda = 1- \frac{1}{2}\frac{   x^* \left(G^* - G\right) x}
	{x^*\left(F+\frac{1}{2}\left(G+G^*\right)\right)x}.
	$$
	Similarly, we can see that  the eigenvalues of $V$ are of the form
	$$
	\lambda = 1+ \frac{1}{2}\frac{   x^* \left(G^* - G\right) x}
	{x^* \left(F+\frac{1}{2}\left(G+G^*\right)\right)x}.
	$$
	Now, having in mind that 
	\begin{eqnarray*}
		x^* \left(F+\frac{1}{2}\left(G+G^*\right)\right)x&>&0,\\
		\Re\left(x^* \left(G^* - G\right) x\right) &=&0, 
	\end{eqnarray*}		
	the desired result is obtained. 
\end{proof}	

\begin{remark}
	As we know the matrix $G$ can be written as $G=H+S$ where 
	$$H=\frac{1}{2}\left(G+G^*\right)\quad 	\textrm{and} \quad
	S=\frac{1}{2}\left(G-G^*\right),$$
	which are the Hermitian and skew-Hermitian part of $G$, respectively. 
	In practice, we assume that  the norm of the  skew-Hermitian part of $G$  is small. In this case the eigenvalues of 
	$\mathbf{R}^{-1} \mathbf{Q}$
	are clustered around $1+0 \textrm{i}$ and we expect that the matrix is well-conditioned.
\end{remark}

Now, let the norm of skew-Hermitian part of $G$ is small and  $\mathbf{R}$ is applied as a preconditioner to the system
\eqref{Eq01}.
In this case, we have
$$\mathbf{R}^{-1} \mathbf{A}=
\left(\mathbf{R}^{-1}\mathbf{Q}\right) \left(\mathbf{Q}^{-1} \mathbf{A}\right).$$ 
Since the eigenvalues of
$\mathbf{R}^{-1} \mathbf{Q}$
are clustered around $1+0\textrm{i}$ and the eigenvalues of
$\left(\mathbf{Q}^{-1} \mathbf{A}\right)$
are clustered in the interval $[0.5,1]$, we expect that the matrix 
$\mathbf{R}^{-1} \mathbf{A}$
is well-conditioned. Note that
$$\textrm{cond
}\left(\mathbf{R}^{-1} \mathbf{A}\right)\leq \textrm{cond}\left(\mathbf{R}^{-1}\mathbf{Q}\right)  \textrm{cond}\left(\mathbf{Q}^{-1} \mathbf{A}\right).$$

\section{A test example}\label{Sec3}

Consider the problem of minimizing the cost functional
\cite{Axe3,Krendl,Liang,Salkuyeh-Calcolo} 
\begin{equation}\label{EqOptim}
	{\cal G}(y,u)=\frac{1}{2}\int_0^T\int_\Omega  |y(x,t)-y_d(x,t)|^2dxdt +\frac{\nu}{2}\int_0^T\int_\Omega |u_c(x,t)|^2dxdt,
\end{equation}
involving computing the state $y(x, t)$ and the controller function $u_c(x, t)$, under the conditions 
\begin{eqnarray}
	\nonumber \frac{\partial}{\partial t}y(x,t)-\Delta y(x,t)&\h=\h&u_c(x,t)~~ \text{in}~~ \Omega \times (0,T),~~~~~~~~~~~~~~~~~~~\label{EqMain}\\
	\nonumber y(x,t) &\h=\h& 0 \hspace{0.75cm}~~~\text{on}~~ \partial Q\times (0,T),\hspace{2.45cm} \\
	\nonumber y(x,0)&\h=\h& y(x,T) ~~\text{on} ~~\partial \Omega,\hspace{2.2cm}\\
	\nonumber u_c(x,0)&\h=\h& u_c(x,T)~~ \text{in}~~ \Omega ,\label{Eq0005} \hspace{2.6cm}
\end{eqnarray}
where $\Omega$ is  bounded open domain  in $\mathbb{R}^d\,(d=1,2,3)$ with the boundary of $\Omega$, $\partial \Omega $,  being Lipschitz-continuous.  Here, we  have $T>0$, and the function $y_d(x,t)$ is the desired state.  In addition, $\nu$ is a regularization parameter. 
Assuming that the desired state $y_d(x,t)$ is time-harmonic, and then  implementing the discretize-first-then-optimize approach, a system of linear equations of the form 
\begin{equation}\label{RC1}
	\mathbf{A}{\bf x} =\mathbf{b},
\end{equation}
is obtained, where 
\begin{equation}\label{RC1}
	\mathbf{A} =
	\left(\begin{array}{cc}
		M & -\sqrt{\nu}(K-\ic\omega M) \\
		\sqrt{\nu}(K+\ic\omega M)      & M 
	\end{array}\right),
\end{equation}
in which real matrix $M$ represents the mass matrix and $K$ symbolizes the discretized negative Laplacian (for more details, see \cite{Axe3,Krendl,Liang,Salkuyeh-Calcolo}).

If we set  
$$F=M~~ \textrm{and}~~G=\sqrt{\nu}(K+\ic\omega M),$$ 
then we see that
\[
G+G^*=\sqrt{\nu}(K+\ic\omega M)+ \sqrt{\nu}(K-\ic\omega M)=2\sqrt{\nu}K,
\]
which is SPD. So the proposed preconditioner can be written as
\begin{equation}\label{RC2}
	\mathbf{R} = \left(\begin{array}{cc}
		M & -\sqrt{\nu}K \\ 
		\sqrt{\nu}K & M + 2\sqrt{\nu}K
	\end{array}\right),
\end{equation}
which is a real matrix. Also, the matrix $\mathbf{Q}$ takes the following form
\begin{equation} \label{Qopt}
	\mathbf{Q}=\left(\begin{array}{cc}
		M & -\sqrt{\nu}(K-\ic\omega M) \\
		\sqrt{\nu}(K+\ic\omega M)      & M+2\sqrt{\nu}K 
	\end{array}\right).
\end{equation}

\begin{theorem}\label{Thm3}
	Let $\lambda$ be an eignevalue of $\mathbf{R}^{-1}\mathbf{Q}$, where $\mathbf{R}$ and $\mathbf{Q}$ are defined in Eqs. \eqref{RC2} and \eqref{Qopt}, respectively. Then $\Re(\lambda)=1$ and     
	\[
	\frac{\sqrt{\nu} \omega}{1+\sqrt{\nu} \lambda_{\max}(S)} \leq    |\Im(\lambda)| \leq \frac{\sqrt{\nu} \omega}{1+\sqrt{\nu} \lambda_{\min}(S)},
	\]
	where  $S=M^{-\frac{1}{2}} K M^{-\frac{1}{2}}$, respectively. 
\end{theorem}	
	\begin{proof}
		We first see that 
		\begin{eqnarray}
			G^*-G &=& \sqrt{\nu}(K-\ic\omega M)- \sqrt{\nu}(K+\ic\omega M)  \nonumber \\
			             &=& -2\ic\sqrt{\nu}\omega M. \label{RC3}
		\end{eqnarray}
		So, using Theorem \ref{Thm2} we have $\Re(\lambda)=1$ and 
		\begin{eqnarray}
			\Im (\lambda) &=& \pm  \sqrt{\nu} \omega \frac{x^*Mx}{x^*Mx+\sqrt{\nu} x^*Kx} \nonumber\\
			&=& \pm  \sqrt{\nu} \omega \frac{1}{1+\sqrt{\nu} \frac{  y^* M^{-\frac{1}{2}}KM^{-\frac{1}{2}}y}{y^*y}} \nonumber\\
			&=& \pm  \sqrt{\nu} \omega \frac{1}{1+\sqrt{\nu} \frac{  y^*Sy}{y^*y}}. \label{RC}
		\end{eqnarray}
		where $y=M^{\frac{1}{2}}x$, and the matrix  $S=M^{-\frac{1}{2}} K M^{-\frac{1}{2}}$ is SPD . Now, using the Courant-Fischer theorem \cite{SaadBook}, we have
		\[
		0\leq\lambda_{\min}(S)\leq \frac{  y^*Sy}{y^*y} \leq \lambda_{\max}(S),
		\]
		which gives the desired result.
	\end{proof}

Theorem \ref{Thm3} shows that for small values of $\sqrt{\nu} \omega$, the eigenvalues of  $\mathbf{R}^{-1}\mathbf{Q}$ are well-clustered around $1+0\textrm{i}$. 

In each iteration of a Krylov subspace method, like GMRES, for solving the preconditioned system 
\[
\mathbf{R}^{-1} \mathbf{A}{\bf x} = \mathbf{R}^{-1} \mathbf{b},
\]
a system of linear equations of the form 
\begin{equation}
	\mathbf{R} 	 \left(\begin{array}{cc}
		r \\ 
		s
	\end{array}\right)
	= \left(\begin{array}{cc}
		M & -\sqrt{\nu}K \\ 
		\sqrt{\nu}K & M + 2\sqrt{\nu}K
	\end{array}\right)
	\left(\begin{array}{cc}
		r\\ 
		s
	\end{array}\right)
	=
	\left(\begin{array}{cc}
		p\\ 
		q
	\end{array}\right),
\end{equation}
should be solved, which can be accomplished using the following algorithm (see e.g. \cite{Axe1,AKM}). 

\medskip

\noindent\textbf{Algorithm 1}. Solution of $\mathbf{R}[r;s]=[p;q]$\\[0.2cm]
1. Solve $(M+\sqrt{\nu}K)(r+s)=p+q$ for $r+s$; \\[0.2cm]
2. Solve $(M+\sqrt{\nu}K)s=q-\sqrt{\nu}K(r+s)$ for $s$; \\[0.2cm]
3. Compute $r=(r+s)-s$.\\

From the above algorithm we see that in each iteration we have to solve two subsystems with coefficient matrix $M+\sqrt{\nu}K$ which is SPD. So these systems can be solved directly using the Cholesky factorization or inexactly using the CG method.  

\section{Numerical experiments}\label{Sec4}

We consider the problem \eqref{EqOptim} in two-dimensional case with $\Omega=(0,1)^2$ and
\begin{equation}\label{Eq20}
	y_d(x,y)=
	\left\{
	\begin{array}{cl}
		(2x-1)^2(2y-1)^2,&\textrm{if}~~(x,y)\in (0,\frac{1}{2})\times (0,\frac{1}{2}), \\
		0,&\textrm{otherwise},
	\end{array} \right.
\end{equation}
and in three-dimensional case with $\Omega=(0,1)^3 $ and
\begin{equation}\label{Eq20}
	y_d(x,y,z)=
	\left\{
	\begin{array}{cl}
		(2x-1)^2(2y-1)^2(2z-1)^2,&\textrm{if}~~(x,y,z)\in (0,\frac{1}{2})\times (0,\frac{1}{2})\times(0,\frac{1}{2}) , \\
		0,&\textrm{otherwise}.
	\end{array} \right.
\end{equation}

The problem is discretized using the bilinear quadrilateral {\bf Q1} finite elements with a uniform mesh \cite{Elman}. To
generate the system \eqref{RC1} we have used the \textsc{Matlab} codes of the paper \cite{Rees} which is available at 
https://www.numerical.rl.ac.uk/people/rees/.
We compare the numerical results of the proposed method (denoted by  $\mathbf{R}_{MPRESB}$) with the 
block diagonal (BD) preconditioner  \cite{Krendl}
\[
\mathbf{P}_{BD} =\left(\begin{array}{cc}
	(1+\omega \sqrt{\nu})M + \sqrt{\nu}K & 0 \\ 
	0 & (1+\omega \sqrt{\nu})M + \sqrt{\nu}K
\end{array}\right), 
\]
and the block alternating splitting (BAS) preconditioner \cite{Zheng}
\[
\mathbf{P}_{BAS} = (1+\alpha) J(\alpha)\left(\begin{array}{cc}
	\alpha M+\sqrt{\nu} K & 0 \\ 
	0 & \alpha M+\sqrt{\nu} K
\end{array}\right),
\] 
in which
\[
J(\alpha)=\frac{1}{\alpha(2+\nu\omega^2)} \left(\begin{array}{cc}
	I &  (1+\nu \omega^2-\ic\omega \sqrt{\nu}) I \\ 
	(1+\nu \omega^2+\ic\omega \sqrt{\nu}) I &  -I
\end{array}\right).
\]
As suggested in \cite{Zheng}, for the BAS preconditioner the parameter $\alpha$ is set to be $$\alpha=\frac{1+\nu \omega^2}{1+\omega \sqrt{\nu}}.$$  
We use the restarted version of GMRES(20) for solving  the system \eqref{RC1} in conjunction with the proposed preconditioner (denoted by ), PRESB preconditioner $\mathbf{Q}$ given in \eqref{Eq05} (hereafter, is denoted by $\mathbf{Q}_{PRESB}$),   $\mathbf{P}_{BD}$ and  $\mathbf{P}_{BAS}$. We utilize right preconditioning.  In the implementation of the  proposed preconditioner, $\mathbf{P}_{BD}$ and  $\mathbf{P}_{BAS}$, two systems with the coefficient matrices $M+\sqrt{\nu}K$, $(1+\omega \sqrt{\nu})M + \sqrt{\nu}K$ and $\alpha M+\sqrt{\nu}K$ should be solved, respectively. 
Since all of these matrices are SPD, we use the sparse Cholesky factorization of the coefficient matrix combined with the symmetric approximate minimum degree  reordering \cite{SaadBook}. For this,  we have utilized the ``\verb"symamd"'' command of \textsc{Matlab}. 
On the other hand, considering \eqref{Eq07a},  in the implementation of the  PRESB preconditioner we should solve two subsystems with the coefficient matrices
	\begin{eqnarray}
		F+G     &=& (1+ i\omega \sqrt{\nu} )M +\sqrt{\nu} K, \\
		F+G^* &=& (1-i\omega \sqrt{\nu} )M +\sqrt{\nu} K. 
	\end{eqnarray} 
	These  systems are solved using the LU factorization in conjunction with the approximate minimum degree  reordering. To do so, we utilize the ``\verb"amd"'' command of \textsc{Matlab}. As we will see shortly, since the LU factors of $ F+G$ and  $F+G^*$ are complex, the elapsed CPU of solving the corresponding systems would be considerably large, particularly in the three dimensional case.

We use a zero vector as an initial guess, and the iteration is terminated as soon as the residual 2-norm   of the system \eqref{RC1} is reduced by a factor of $10^8$. A maximum number of 1000 iterations is set to terminate the iterations. In the tables, a dagger ($\dag$) shows that the GMRES method has not converged in the maximum number of iterations.

All the numerical results have been computed by some  \textsc{Matlab} codes   with a Laptop with 2.40 GHz central processing unit (Intel(R) Core(TM) i7-5500), 8 GB memory and Windows 10 operating system.

We first investigate the problem in two dimensional case.  Figure \ref{Fig1} displays the eigenvalue distribution of the matrix $\mathbf{R}_{MPRESB}^{-1}\mathbf{A}$ for $h=2^{-4}$ (the size of the matrix $\mathbf{A}$ is 450) and $\omega=10$ for  $\nu=10^{-k}$, $k=2,4,6,8$. As we see the eigenvalues of the preconditioned matrix are well-clustered.

Numerical results for  $h=2^{-7},2^{-8},2^{-9}$ are presented in Tables \ref{Tbl1}--\ref{Tbl3} for different values of $\omega$. In this case,  the size of  the matrices are $32258$, $130050$ and $522242$, respectively.   In each table we present the number of iterations for the convergence and  the elapsed CPU time (in second and in parenthesis). As the numerical results show when $\sqrt{\nu}\omega$ is small enough ($\sqrt{\nu}\omega \leq 10$), the MPRESB method is is more efficient than the others. For $\sqrt{\nu}\omega \leq 10$, the MPRESB can not competes with the other preconditioners.  Another observation which can be posed here  is that for $\sqrt{\nu}\omega \leq 10$, the number of iterations of MPRESB is close to that of PRESB.

We now present the numerical results in three dimensional case. We set $h=2^{-4},2^{-5}$ and the numerical experiments are given in Tables \ref{Tbl4} and \ref{Tbl5}. In this case,  the size of  the matrices are $6750$ and $59582$. The numerical results show that all the observations in two dimensional can  also be expressed here. In addition, as we see the elapsed CPU time of MPRESB are very large for $h=2^{-5}$. Apparently, this is due to the computation of  the LU   factorization of the matrices $F+G$ and $F+G^*$ and their implementations for solving the subsystems.  

\begin{figure}
	\centering
	\includegraphics[height=7cm,width=8.cm]{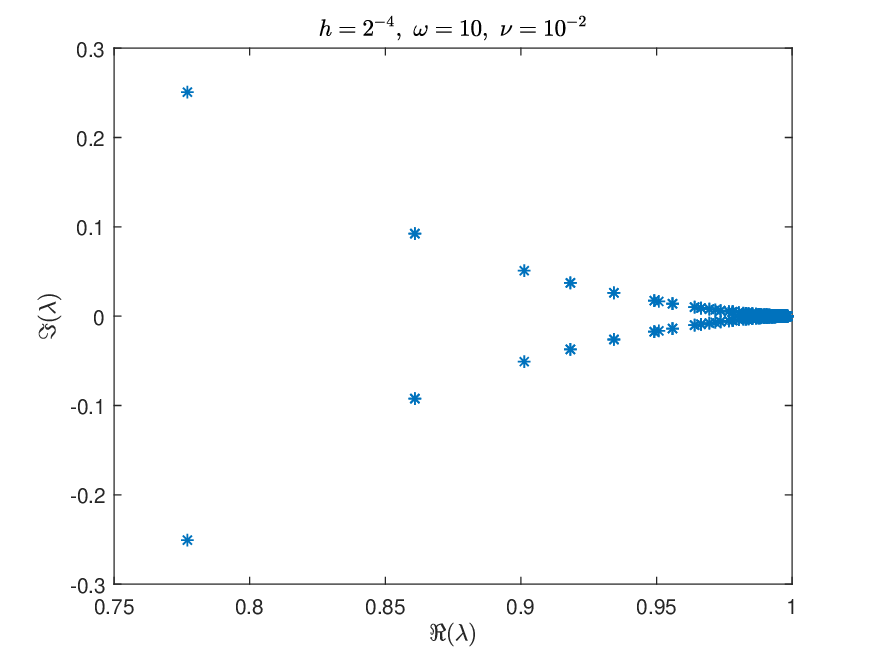}~~~\includegraphics[height=7cm,width=8cm]{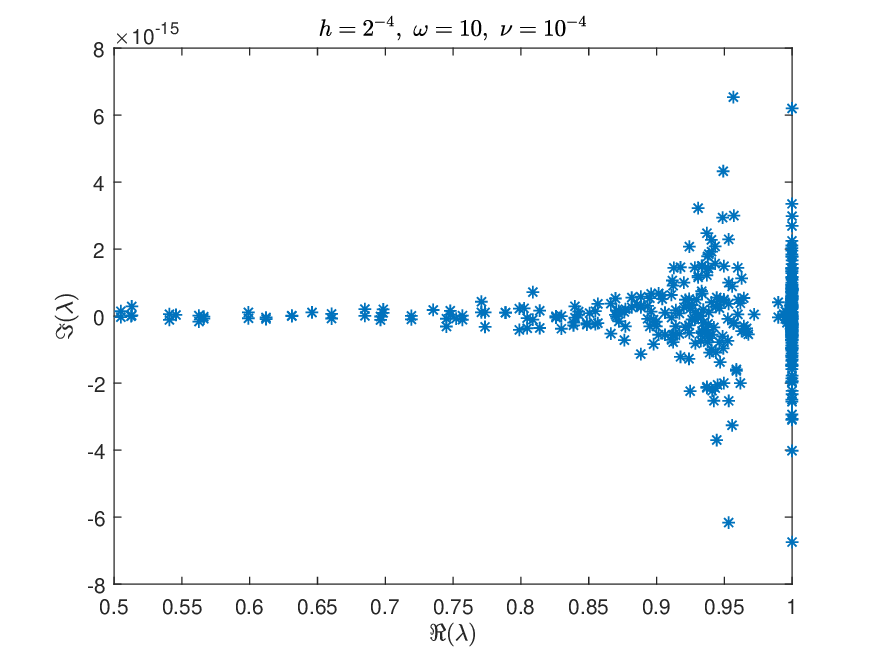} \\
	\includegraphics[height=7cm,width=8.cm]{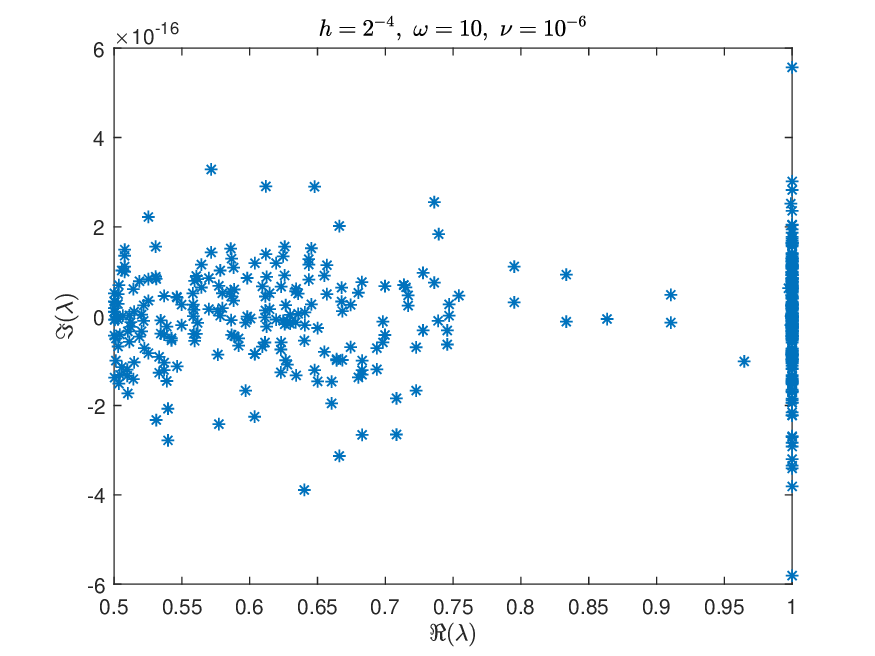}~~~\includegraphics[height=7cm,width=8cm]{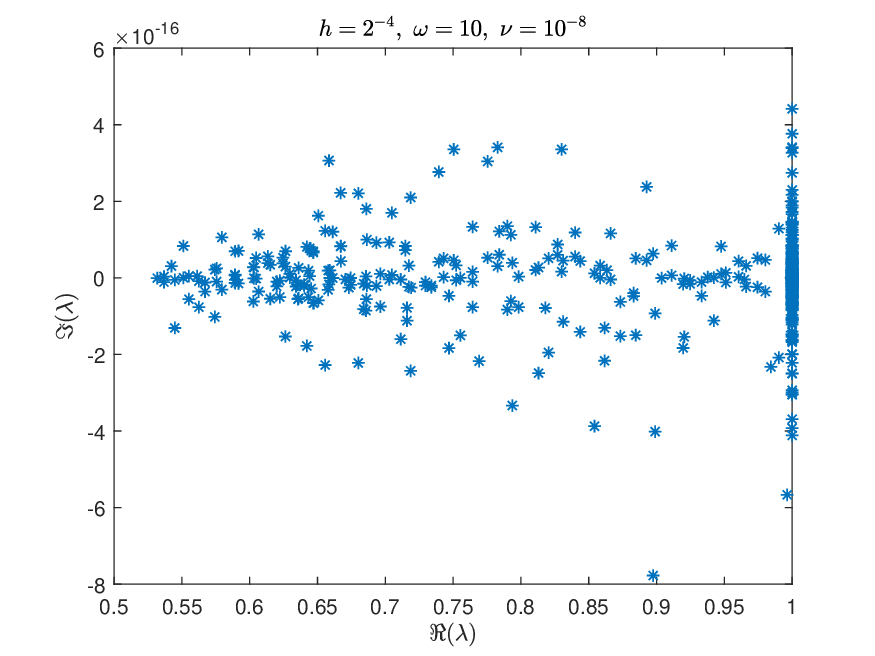} \\
	\caption{Eigenvalue distribution of  $\mathbf{R}_{MPRESB}^{-1}\mathbf{A}$ in two dimensional case for $h=2^{-4}$ and $\omega=10$ for  $\nu=10^{-k}$, $k=2,4,6,8$.	\label{Fig1}}
\end{figure}

\begin{table}[!htp]
	\centering\caption{Number of iterations of the methods along with the elapsed CPU time (in parenthesis) in two dimensional case  for $h=2^{-7}$ and different values of $\nu$ and $\omega$. \label{Tbl1}}
	\begin{center}
		\label{exact}
		%\scriptsize
		\scalebox{1.}{
			\begin{tabular}{|c|c|c|c|c|c|c|c|c|c|} \hline
				&$ \nu\setminus\omega$    & $10^{-2}$  & $10^{-1}$ & $1$  & $10^{1}$ & $10^{2}$  & $10^3$ & $10^4$\\ \hline			
				%				No Prec.     & $10^{-2}$   & $\dag$      & $\dag$        & $\dag$      & $\dag$     & $\dag$  \\
				%				& $10^{-4}$   & $\dag$      & $\dag$        & $\dag$      & $\dag$      & $\dag$  \\
				%				& $10^{-6}$   & 1303(14.79)      & 1303(14.76)        & 1303(14.76)      & 1303(14.69)      &  1303(14.74)  \\
				%				& $10^{-8}$   & 146(1.84) & 146(1.87)   & 146(1.89) & 146(1.91) & 146(1.86)  \\     \hline
				
				$\mathbf{R}_{MPRESB}$       & $10^{-2}$   & 9(0.36)        & 9(0.35)         & 9(0.37)     & 10 (0.40)  & 24(0.79)  & 246(6.37) & $\dag$\\
				& $10^{-4}$   & 12(0.45)       & 12(0.44)      & 12(0.44)   & 12(0.44)    & 18(0.63)  & 139(3.66) &$\dag$\\
				& $10^{-6}$   & 12(0.44)       & 12(0.45)      & 12(0.44)   & 12(0.43)    & 12(0.46)  &27(0.83)& 248(6.29)\\
				& $10^{-8}$   & 11(0.40)       & 11(0.41)      & 11(0.43)  & 11(0.42)     & 11(0.44)  &12(0.42)& 27(0.80)\\ \hline

				$\mathbf{Q}_{PRESB}$       & $10^{-2}$   & 9(0.83)        & 9(0.78)         & 9(0.75)     & 9 (0.76)     & 7(0.70)    &5(0.67)   &4 (0.64)  \\
				& $10^{-4}$   & 12(0.86)       & 12(0.87)      & 12(0.86)   & 12(0.86)    & 11(0.84)  &6(0.70)   &4(0.68)   \\
				& $10^{-6}$   & 12(0.84)       & 12(0.85)      & 12(0.83)   & 12(0.84)    & 12(0.86)  &10(0.79) &6(0.70) \\
				& $10^{-8}$   & 11(0.80)       & 11(0.85)      & 11(0.81)  & 11(0.82)     & 11(0.81)  &11(0.82) &10(0.77) \\ \hline

				$\mathbf{P}_{BD}$   & $10^{-2}$   & 20(0.58)         & 20(0.61)        & 20(0.67)     & 22(0.67)     & 26(0.74)  & 22(0.64)&22(0.64) \\
				& $10^{-4}$   & 56(1.45)         & 56(1.44)        & 56(1.49)     & 58(1.56)     & 48(1.28)  & 26(0.69)&22(0.65) \\
				& $10^{-6}$   & 61(1.61)         & 61(1.64)        & 61(1.61)     & 61(1.61)     & 62(1.63)  & 50(1.31)&24(0.63) \\
				& $10^{-8}$    & 54(1.41)         & 54(1.36)        & 54(1.41)     & 54(1.46)     & 54(1.42) & 54(1.45) &44(1.24)\\    \hline

				$\mathbf{P}_{BAS}$  & $10^{-2}$   & 16(0.48)         & 16(0.48)        & 16(0.50)     & 18(0.52)     & 54(1.43)  & 52(1.59) & 43(1.25)\\
				& $10^{-4}$   & 22(0.66)         & 22(0.68)        & 22(0.67)     & 22(0.67)     & 50(1.36)  & 65(1.62) & 47(1.19)\\
				& $10^{-6}$   & 22(0.68)         & 22(0.65)        & 22(0.68)     & 22(0.68)     & 26(0.73)  & 70(1.84) & 57(1.45)\\
				& $10^{-8}$   & 22(0.65)         & 22(0.67)        & 22(0.67)     & 22(0.66)     & 22(0.66)  &26(0.75) & 64 (1.70)\\
				\hline
				
		\end{tabular}}
	\end{center}
\end{table}

\begin{table}[!htp]
	\centering\caption{Number of iterations of the methods along with the elapsed CPU time (in parenthesis) in two dimensional case  for $h=2^{-8}$ and different values of $\nu$ and $\omega$. \label{Tbl2}}
	\begin{center}
		\label{exact}
		%\scriptsize
		\scalebox{1.}{
			\begin{tabular}{|c|c|c|c|c|c|c|c|c|c|} \hline
				&$\nu\setminus\omega$    & $10^{-2}$  & $10^{-1}$ & $1$  & $10^{1}$ & $10^{2}$ & $10^3$ & $10^4$\\ \hline			
				%				No Prec.     & $10^{-2}$   & $\dag$      & $\dag$        & $\dag$      & $\dag$     & $\dag$  \\
				%				& $10^{-4}$   & $\dag$      & $\dag$        & $\dag$      & $\dag$      & $\dag$  \\
				%				& $10^{-6}$   & $\dag$      & $\dag$        & $\dag$      & $\dag$      & $\dag$  \\
				%				& $10^{-8}$   & 523(26.15) & 523(25.77)  & 523(25.70) & 523(25.86) & 523(25.58)  \\     \hline

				$\mathbf{R}_{MPRESB}$       & $10^{-2}$   & 9(1.30)        & 9(1.30)         & 9(1.32)     & 10 (1.39)  & 24(3.12) & 251(27.84)& $\dag$\\
				& $10^{-4}$   & 12(1.62)       & 12(1.62)      & 12(1.62)   & 12(1.66)    & 18(2.41) & 139(16.42)& $\dag$ \\
				& $10^{-6}$   & 12(1.63)       & 12(1.64)      & 12(1.62)   & 12(1.63)    & 12(1.67) & 27(3.39)   & 254(28.25) \\
				& $10^{-8}$   & 11(1.51)       & 11(1.49)      & 11(1.54)  & 11(1.50)     & 11(1.53) & 12(1.58) & 27(3.35) \\ \hline
				
				$\mathbf{Q}_{PRESB}$         & $10^{-2}$   & 9(6.05)        & 9(6.20)         & 9(6.11)     & 9 (6.08)    & 7(5.97)   & 5(5.79)& 4(5.55)\\
				& $10^{-4}$   & 12(6.50)       & 12(6.56)      & 12(6.38)   & 12(6.45)    & 11(6.23) & 7(5.90)& 4 (5.60)\\
				& $10^{-6}$   & 12(6.39)       & 12(6.43)      & 12(6.29)   & 12(6.40)    & 12(6.34) & 11(6.19)& 6(5.77) \\
				& $10^{-8}$   & 11(6.27)       & 11(6.37)      & 11(6.24)  & 11(6.36)     & 11(6.32) &11(6.27)& 10(6.23)\\ \hline
				
				$\mathbf{P}_{BD}$  & $10^{-2}$   & 20(2.57)         & 20(2.58)        & 20(2.57)     & 22(2.84)     & 26(3.09)   &22(2.72)& 22(2.74)\\
				& $10^{-4}$   & 56(6.32)          & 56(6.39)        & 56(6.41)     & 58(6.67)     & 49(5.63)  &26(3.02)& 22(2.75)\\
				& $10^{-6}$   & 62(7.16)          & 62(7.18)        & 62(7.12)     & 62(7.19)     & 63(7.18)  &50(5.67)& 24(2.89) \\
				& $10^{-8}$    & 57(7.10)         & 57(6.50)        & 57(6.56)     & 57(6.62)     & 57(6.51)  &58(6.67)& 48(5.46)\\    \hline

				$\mathbf{P}_{BAS}$   & $10^{-2}$   & 16(2.03)         & 16(2.05)        & 16(2.05)     & 18(2.31)     & 54(6.26)  &53(5.91)& 45(5.17) \\
				& $10^{-4}$   & 22(2.83)         & 22(2.84)        & 22(2.83)     & 22(2.83)     & 50(5.75)  &65(7.33)& 50(5.63)\\
				& $10^{-6}$   & 22(2.90)         & 22(2.83)        & 22(2.84)     & 22(2.83)     & 26(3.17)   &71(8.05)& 62(7.02)\\
				& $10^{-8}$   & 22(2.98)         & 22(2.82)        & 22(2.84)     & 22(2.84)     & 22(2.85)   &26(3.18)& 68(7.69)\\
				\hline

		\end{tabular}}
	\end{center}
\end{table}

\begin{table}[!htp]
	\centering\caption{Number of iterations of the methods along with the elapsed CPU time (in parenthesis) in two dimensional case  for $h=2^{-9}$ and different values of $\nu$ and $\omega$. \label{Tbl3}}
	\begin{center}
		\label{exact}
		%\scriptsize
		\scalebox{1.}{
			\begin{tabular}{|c|c|c|c|c|c|c|c|c|c|} \hline
				&$\nu\setminus\omega$    & $10^{-2}$  & $10^{-1}$ & $1$  & $10^{1}$ & $10^{2}$ & $10^{3}$ &  $10^{4}$\\ \hline			
				%				No Prec.     & $10^{-2}$   & $\dag$      & $\dag$        & $\dag$      & $\dag$     & $\dag$  \\
				%				& $10^{-4}$   & $\dag$      & $\dag$        & $\dag$      & $\dag$      & $\dag$  \\
				%				& $10^{-6}$   & $\dag$      & $\dag$        & $\dag$      & $\dag$      & $\dag$  \\
				%				& $10^{-6}$   & $\dag$      & $\dag$        & $\dag$      & $\dag$      & $\dag$  \\     \hline

				$\mathbf{R}_{MPRESB}$   & $10^{-2}$   & 9(6.33)    & 9(6.51)   & 9(6.38)     & 10 (6.39)    & 24(15.73)  & 252(148.05) & $\dag$\\
				& $10^{-4}$   & 12(8.09)  & 12(8.15) & 12(8.10)   & 12(8.20)     & 18(12.07)  &139(83.42) &$\dag$\\
				& $10^{-6}$   & 12(8.14)  & 12(8.18) & 12(8.12)   & 12(8.12)     & 12(8.14)    & 27(17.41) &257(150.03)\\
				& $10^{-8}$   & 11(7.54)  & 11(7.49) & 11(7.56)   & 11(7.48)     & 11(7.56)    &12(8.13)&  27(17.41) \\ \hline
				
				$\mathbf{Q}_{PRESB}$ & $10^{-2}$   & 9(52.04)        & 9(51.67)         & 9(52.82)     & 9 (53.02)    & 7(52.15)    & 5(50.82)& 4(50.37)\\
				& $10^{-4}$   & 12(54.34)       & 12(54.42)      & 12(56.36)   & 12(57.05)    & 11(54.96) & 7(52.64)& 5 (52.24)\\
				& $10^{-6}$   & 12(53.97)       & 12(53.42)      & 12(55.47)   & 12(54.81)    & 12(55.84) & 11(53.92)& 6(51.48) \\
				& $10^{-8}$   & 11(52.24)       & 11(52.53)      & 11(53.80)  & 11(55.01)     & 11(54.25) &11(55.13)& 10(53.66)\\ \hline

				$\mathbf{P}_{BD}$  & $10^{-2}$   & 20(13.43)          & 20(13.36)        & 20(13.32)     & 22(14.63)     & 26(16.39) &22(14.11)& 22(14.15) \\
				& $10^{-4}$   & 56(33.90)          & 56(33.78)        & 56(33.94)     & 58(35.37)     & 49(29.46) &26(16.07)& 22(14.70) \\
				& $10^{-6}$   & 62(38.07)          & 62(37.55)        & 62(37.60)     & 62(37.75)     & 63(38.02) &50(30.02)&26(16.03) \\
				& $10^{-8}$    & 59(35.88)         & 59(36.29)        & 59(35.74)     & 59(35.77)     & 59(35.48) &60(36.46)& 48(29.15)\\    \hline

				$\mathbf{P}_{BAS}$   & $10^{-2}$   & 16(10.68)         & 16(10.69)        & 16(10.71)     & 18(12.05)     & 54(32.83)&53(31.54)&47(27.85) \\
				& $10^{-4}$   & 22(14.83)         & 22(14.79)        & 22(14.87)     & 22(14.72)     & 50(30.14)&65(38.43)& 51(31.12)\\
				& $10^{-6}$   & 22(15.18)         & 22(14.73)        & 22(14.85)     & 22(14.87)     & 26(16.52)&72(42.44)& 63(37.91) \\
				& $10^{-8}$   & 22(14.73)         & 22(14.85)        & 22(14.77)     & 22(14.85)     & 22(14.75)&26(16.62)& 70(41.43)  \\
				\hline

		\end{tabular}}
	\end{center}
\end{table}

\begin{table}[!htp]
	\centering\caption{Number of iterations of the methods along with the elapsed CPU time (in parenthesis) in three dimensional case  for $h=2^{-4}$ and different values of $\nu$ and $\omega$. \label{Tbl4}}
	\begin{center}
		\label{exact}
		%\scriptsize
		\scalebox{1.}{
			\begin{tabular}{|c|c|c|c|c|c|c|c|c|c|} \hline
				&$\nu\setminus\omega$    & $10^{-2}$  & $10^{-1}$ & $1$  & $10^{1}$ & $10^{2}$ & $10^{3}$ &  $10^{4}$\\ \hline			
				
				$\mathbf{R}_{MPRESB}$   
				& $10^{-2}$   & 9(0.19)    & 9(0.18)   & 9(0.17)     & 9 (0.18)    & 25(0.39)    & 239(2.93 )& $\dag$\\
				& $10^{-4}$   & 12(0.20)  & 12(0.20) & 12(0.21)   & 12(0.19)     & 18(0.29)  & 128(1.53) &$\dag$\\
				& $10^{-6}$   & 10(0.19)  & 10(0.19) & 10(0.18)   & 10(0.19)     & 12(0.20)  &26(0.39)    &106(1.32)\\
				& $10^{-8}$  & 8(0.15)    & 8(0.16)   & 8(0.15)      & 9(0.16)     & 9(0.18)     & 10(0.16)    & 17(0.29) \\ \hline
				
				$\mathbf{Q}_{PRESB}$ 
				& $10^{-2}$   & 9(1.38)     & 9(1.33)     & 9(1.35)     & 9 (1.33)     & 7(1.30)    & 4(1.34)& 3(1.48)\\
				& $10^{-4}$   & 12(1.39)   & 12(1.37)    & 12(1.46)   & 12(1.38)    & 10(1.40) & 6(1.31)& 4(1.37)\\
				& $10^{-6}$   & 10(1.33)   & 10(1.34)    & 10(1.37)   & 10(1.35)    & 10(1.33) & 10(1.36)& 5(1.44) \\
				& $10^{-8}$   & 8(1.33)     & 8(1.32)      & 8(1.35)     & 8(1.32)      & 9(1.38)   &9(1.37)&  8(1.43)\\ \hline

				$\mathbf{P}_{BD}$  
				& $10^{-2}$   & 18(0.28)   & 18(0.26)     & 18(0.27)     & 20(0.30) &24(0.35)&  20(0.30)& 18(0.28)\\
				& $10^{-4}$   & 51(0.64)   & 51(0.64)     & 52(0.63)     & 52(0.63) &44(0.55)& 20(0.30) & 18(0.26)\\
				& $10^{-6}$   & 45(0.58)   & 45(0.56)     & 45(0.60)     & 45(0.56) &45(0.57)& 34(0.44) & 18(0.29)\\
				& $10^{-8}$   & 13(0.19)   & 13(0.19)     & 13(0.20)     & 13(0.21) &13(0.19)& 13(0.18) &14(0.22)\\    \hline

				$\mathbf{P}_{BAS}$   
				& $10^{-2}$   & 16(0.26)   & 16(0.23)        & 15(0.22)     & 16(0.24)     & 48(0.62)&39(0.49)& 11(0.18)\\
				& $10^{-4}$   & 21(0.33)   & 21(0.32)        & 20(0.30)     & 22(0.32)     & 44(0.56)&48(0.62)& 12(0.20)\\
				& $10^{-6}$   & 20(0.30)   & 20(0.29)        & 20(0.29)     & 20(0.29)     & 24(0.34)&54(0.68)&  15(0.22)\\
				& $10^{-8}$   & 18(0.26)   & 18(0.26)        & 18(0.26)     & 18(0.27)     & 18(0.26)&20(0.27)&  22(0.34)\\
				\hline

		\end{tabular}}
	\end{center}
\end{table}

\begin{table}[!htp]
	\centering\caption{3D:Number of iterations of the methods along with the elapsed CPU time (in parenthesis) in three dimensional case for $h=2^{-5}$ and different values of $\nu$ and $\omega$. \label{Tbl5}}
	\begin{center}
		\label{exact}
		%\scriptsize
		\scalebox{1.}{
			\begin{tabular}{|c|c|c|c|c|c|c|c|c|c|} \hline
				&$\nu\setminus\omega$    & $10^{-2}$  & $10^{-1}$ & $1$  & $10^{1}$ & $10^{2}$ & $10^{3}$ &  $10^{4}$\\ \hline			
				
				$\mathbf{R}_{MPRESB}$   
				& $10^{-2}$   & 9(2.86)       & 9(2.76)   & 9(2.73)     & 9 (2.74)    & 25(5.81)    & 268(49.82 )&  $\dag$\\
				& $10^{-4}$   & 12(3.29)     & 12(3.28) & 12(3.28)   & 12(3.26)   & 18(4.40)    & 137(26.37) & $\dag$\\
				& $10^{-6}$   & 11(3.08)     & 11(3.15) & 11(3.69)   & 11(3.17)   & 12(3.33)    &27(6.17)      &190(36.17)\\
				& $10^{-8}$  & 10(2.89)      & 10(3.15) & 10(4.04)   & 10(3.73)  & 10(3.65)     & 11(3.97)     &24(5.50)\\ \hline
				
				$\mathbf{Q}_{PRESB}$ 
				& $10^{-2}$   & 9(128.12)    & 9(128.80)   & 9(128.27)     & 9 (129.17)    & 7(128.08)    & 4(126.29 )& 4(128.45)\\
				& $10^{-4}$   & 12(130.55)  & 12(128.11) & 12(128.74)   & 12(130.36)   & 11(129.76)  & 6(131.46) &4(130.79)\\
				& $10^{-6}$   & 11(130.32)  & 11(127.86) & 11(129.06)   & 11(130.14)   & 11(127.94)  &10(128.83) & 6(127.89)\\
				& $10^{-8}$  & 10(127.88)   & 10(132.48) & 10(129.90)   & 10(129.86)  & 10(131.00)   & 10(129.99)  &9(129.71)\\ \hline

				$\mathbf{P}_{BD}$  
				& $10^{-2}$   & 18(4.34)     & 18(4.36)       & 18(4.40)      & 20(4.81)   & 24(5.63) &  20(4.64)& 22(5.31) \\
				& $10^{-4}$   & 52(10.62)   & 52(10.44)     & 53(10.89)    & 54(10.79) & 46(9.53) & 24(5.47) & 22(5.25) \\
				& $10^{-6}$   & 55(10.98)   & 55(11.06)     & 55(11.06)    & 55(11.21) &56(11.22)& 44(9.10) & 20(4.68)\\
				& $10^{-8}$   & 30(6.64)     & 30(6.60)       & 30(7.19)      & 30(8.03)   &30(8.04) & 30(7.96) & 22(5.09)\\    \hline

				$\mathbf{P}_{BAS}$   
				& $10^{-2}$   & 16(4.00)   & 16(4.07)        & 15(3.88)     & 16(4.00)     & 50(10.34) & 48(9.81)& 23(5.44)\\
				& $10^{-4}$   & 22(5.43)   & 22(5.36)        & 22(5.25)     & 22(5.30)     & 44(9.22)  & 58(11.77)& 26(5.95)\\
				& $10^{-6}$   & 20(4.75)   & 20(4.83)        & 20(4.86)     & 20(4.82)     & 24(5.76)  & 64(12.89)& 32(6.97) \\
				& $10^{-8}$   & 22(5.25)   & 22(5.30)        & 22(5.52)     & 22(6.49)     & 22(6.43)  & 24(6.79)&  38(7.88)\\
				\hline

		\end{tabular}}
	\end{center}
\end{table}

\section{Conclusion}\label{Sec5}

We have modified the PRESB preconditioner to  a class of  non-Hermitian  two-by-two complex system of linear equations. Spectral analysis of the preconditioned matrix has been analyzed. We utilized the proposed preconditioner to the systems arising from finite element discretization of  PDE-constrained optimization problems. Numerical results show that the proposed preconditioner is efficient and in comparison to two well-known precoditioners is more efficient.

\section*{Confict of interest}

The authors declare that there is  no conflict of interest.

\section*{Acknowledgements}

The authors would like to thank Prof. J\'{a}nos Kar\'{a}tson (E\"{o}tv\"{a}s Lor\'{a}nd University, Hungary) for his helpful comments and suggestions on an earlier draft of the paper.


\begin{thebibliography}{}
\bibitem{Arridge} S.R. Arridge, Optical tomography in medical imaging, \textit{Inverse Problems} 15 (1999) 41--93.

\bibitem{Axe1}  O. Axelsson, S.  Farouq, and M. Neytcheva, Comparison of preconditioned Krylov subspace iteration
methods for PDE-constrained optimization problems. Poisson and convection-diffusion control,
\textit{Numer. Algorithms},  73 (2016) 631--663.

\bibitem{Axe2} O. Axelsson, S. Farouq, M. Neytcheva,  Comparison of preconditioned Krylov subspace iteration
methods for PDE-constrained optimization problems. Stokes Control,  \textit{Numer. Algorithms} 74 (2017) 19--37.

\bibitem{Axe3} O. Axelsson, S. Farouq, M. Neytcheva,  A preconditioner for optimal control problems constrained
by Stokes equation with a time-harmonic control,  \textit{J. Comput. Appl. Math.}  310 (2017)  5--18.

\bibitem{Axe4} O. Axelsson, D.K. Salkuyeh, A new version of a preconditioning method for certain two-by-two block matrices with square blocks, \textit{BIT Numer. Math.} 59 (2019) 321--342.


\bibitem{Lukas} O. Axelsson,  D. Luk$\rm\acute{a}\breve{s}$, Preconditioning methods for eddy-current optimally controlled time-harmonic electromagnetic problems, \textit{J. Numer. Math.} 27 (2019) 1-21. 

\bibitem{Owe-Janos} O. Axelsson,   J. Kar\'{a}tson, Superior properties of the PRESB preconditioner for
operators on two-by-two block form with square blocks, \textit{Numer. Math.} 146 (2020) 335--368.


\bibitem{AxK} O. Axelsson, J. Kar\'{a}tson, Equivalent operator preconditioning for elliptic problems,
\textit{Numer. Algorithms} 50 (2009) 297--380.


\bibitem{AKM} O. Axelsson, J. Kar\'{a}tson, F. Magoul\`{e}s,  Superlinear convergence using block precon-
ditioners for the real system formulation of complex Helmholtz equations, \textit{J. Comput. Appl. Math.} 340 (2018) 424--431.


\bibitem{AKN} O. Axelsson, J. Kar\'{a}tson, M. Neytcheva,  Preconditioned iterative solution methods for
linear systems arising in PDE-constrained optimization, in: Robust and Constrained Optimization: Methods and Applications. Nova Science Publishers, pp. 85--148.

\bibitem{BaiBook}  Z.-Z. Bai, J.-Y. Pan, Matrix Analysis and Computations, SIAM, Philadelphia, 2021.


\bibitem{Bai1} Z.-Z. Bai, G.H. Golub, M.K. Ng, Hermitian and skew-Hermitian splitting methods for non-Hermitian positive definite linear systems, \textit{SIAM. J. Matrix Anal. Appl.} 24 (2003) 603--626.

\bibitem{Bai2} Z.-Z. Bai, M. Benzi, F. Chen, Modified HSS iteration methods for a class of complex symmetric linear systems, \textit{Computing} 87 (2010) 93--111.

\bibitem{Bai3}Z.-Z. Bai, M. Benzi, F. Chen, On preconditioned MHSS iteration methods for complex symmetric linear systems, \textit{Numer. Algorithms} 56 (2011) 297--317.

\bibitem{Bertaccini} D. Bertaccini, Efficient solvers for sequences of complex symmetric linear system, \textit{Electron. Trans. Numer. Anal.} 18 (2004) 49--64.

\bibitem{BenziFV} M. Benzi, Some uses of the field of values in numerical analysis, Bolletino dell Unione Mat. Ital.  14 (2021) 159--177.

\bibitem{BenziJCP} M. Benzi, Preconditioning techniques for large linear systems: a survey, J. Comput. Phys. 182  (2002) 418--477.


\bibitem{EES} S.C. Eisenstat, H.C. Elman, M.H. Schultz,  Variational iterative methods for nonsymmetric systems of linear equations, SIAM J. Numer. Anal.  20 (1983) 345--357.


\bibitem{Elman} H. Elman, D. Silvester, A.J. Wathen, Finite Elements and Fast Iterative Solvers with Applications in Incompressible Fluid Dynamics, Oxford University Press, 2005.


\bibitem{Feriani} A. Feriani, F. Perotti, V. Simoncini, Iterative system solvers for the frequency analysis of linear mechanical systems, \textit{Comput. Methods Appl. Mech. Engrg.} 190 (2000) 1719--1739.


\bibitem{Frommer} A. Frommer, T. Lippert, B. Medeke, K. Schilling, Numerical challenges in lattice quantum chromodynamics, \textit{Lecture notes in computational science and engineering} 15 (2000) 1719--1739.

\bibitem{Hezari-NLA} D. Hezari, V. Edalatpour, D.K. Salkuyeh, Preconditioned GSOR iterative method for a class of complex symmetric system of linear equations, \textit{Numer. Linear Algebra  Appl.} 22 (4) 761--776.

\bibitem{SCSP} D. Hezari, D.K. Salkuyeh, V. Edalatpour, A new iterative method for solving a class of complex symmetric system of linear equations, \textit{Numer. Algorithms} 73 (2016) 927--955

\bibitem{Krendl}  W. Krendl, V. Simoncini, W. Zulehner, Stability estimates and structural spectral properties of saddle point problems, \textit{Numer. Math.} {\bf 124} (2013), 183--213.

\bibitem{Liang} Z.-Z. Liang, O. Axelsson, M. Neytcheva, A robust structured preconditioner for time-harmonic
parabolic optimal control problems, \textit{Numer. Algorithms}  79 (2018) 575--596.

\bibitem{Moro} G. Moro, J.H. Freed, Calculation of ESR spectra and related Fokker-Planck forms by the use of the Lanczos algorithm, \textit{J. Chem. Phys.} 74 (1981) 3757--3773.

\bibitem{Poirier} B. Poirier, Effecient preconditioning scheme for block partitioned matrices with structured sparsity, \textit{Numer. Linear Algebra Appl.} 7 (2000) 715--726.

\bibitem{Rees} T. Rees, H.S. Dollar, A.J. Wathen, Optimal solvers for PDE-constrained optimization, \textit{SIAM J. Sci. Comput.} 32 (2010)
271--298.

\bibitem{SaadBook} Y. Saad,  \textit{Iterative Methods for Sparse Linear Systems}, PWS Press, New York, 1995.

\bibitem{GMRES} Y. Saad, M.H Schultz, GMRES: a generalized minimal residual algorithm for solving nonsymmetric
linear systems, \textit{SIAM J. Sci. Stat. Comput.}, 7 (1986), 856--869.

\bibitem{Salkuyeh1} D.K. Salkuyeh, D. Hezari, V. Edalatpour, Generalized successive overrelaxation iterative method for a class of complex symmetric linear system of equations, \textit{Int. J. Comput. Math.} 92 (2015) 802--815.

\bibitem{TSCSP} D.K. Salkuyeh, Two-step scale-splitting method for solving complex symmetric system of linear equations, arXiv:1705.02468.

\bibitem{TTSCSP} D.K. Salkuyeh, T.S. Siahkolaei,  Two-parameter TSCSP method for solving complex symmetric system
of linear equations, \textit{Calcolo} 55 (2018) 1--22.

\bibitem{Salkuyeh-Calcolo} D.K. Salkuyeh, A new iterative method for solving a class of two‑by‑two block complex linear systems, Calcolo 58 (2021) 42. 

\bibitem{Schmitt} D. Schmitt, B. Steffen, T. Weiland, 2D and 3D computations of lossy eigenvalue problems, \textit{IEEE Trans. Magn.} 30 (1994) 3578-3581.

\bibitem{CRI} T. Wang, Q.-Q. Zheng, L.-Z. Lu,  A new iteration method for a class of complex symmetric linear
systems. \textit{J. Comput. Appl. Math.} 325 (2017) 188--197.


\bibitem{Zeng} M.-L. Zeng, Respectively scaled splitting iteration method for a class of block 4-by-4 linear systems from eddy current
electromagnetic problems, \textit{Japan J.  Indust. Appl. Math.} 38 (2021) 489--501.


\bibitem{Zheng} Z. Zheng, G.-F. Zhang, M.-Z. Zhu, A block alternating splitting iteration method for a class of block two-by-two complex linear systems, \textit{J. Comput. Appl. Math.} 288 (2015) 203--214.

\end{thebibliography}
\end{document}